\documentclass[12pt]{amsart}
\usepackage{amsthm,amsopn,amsfonts,amssymb,amsmath}
\usepackage[bookmarks, bookmarksdepth=2, colorlinks=true, linkcolor=blue, citecolor=blue, urlcolor=blue]{hyperref}
\usepackage{color,cancel,mathtools,kbordermatrix}
\newcommand{\arxiv}[1]{\href{http://arxiv.org/abs/#1}{{\tt arXiv:#1}}}

\usepackage{enumerate}
\usepackage{euler, amsfonts, amssymb, latexsym, epsfig}

\usepackage{tikz-cd}
\usetikzlibrary{calc}
\usetikzlibrary{arrows,matrix,positioning}

\usepackage[margin=1in]{geometry}

\font\co=lcircle10
\def\boxcross{\ \smash{\lower6.5pt\hbox{\rlap{\hskip4.5pt\vrule height13.5pt}}
                \raise0pt\hbox{\rlap{\hskip-2pt \vrule height.4pt depth0pt
                        width13.5pt}}}\hskip12.7pt}

\def\boxelbow{\ \hskip.1pt\smash{%
               \hbox{\co \hskip 5.5pt\rlap{\mathsurround=0pt\rlap{\mathsurround=0pt\char'006}\lower0.4pt\rlap{\char'004}}
                \lower6.5pt\rlap{\hskip-0.2pt\vrule height3pt}
                \raise3.5pt\rlap{\hskip-0.2pt\vrule height3.2pt}}
                \hbox{%
                  \rlap{\hskip-6.4pt \vrule height.4pt depth0pt
width2.5pt}%
                  \rlap{\hskip4.05pt \vrule height.4pt depth0pt
width3.1pt}}}
                \hskip8.7pt}

\newtheorem{theorem}{Theorem}[section]
\newtheorem{Theorem}[theorem]{Theorem}
\newtheorem*{thm*}{Theorem}
\newtheorem{Lemma}[theorem]{Lemma}
\newtheorem{Proposition}[theorem]{Proposition}

\newtheorem*{Conjecture*}{Conjecture}

\theoremstyle{definition}

\theoremstyle{remark}

\newtheorem{remark}[theorem]{Remark}

\newcommand\onto{\twoheadrightarrow}

\newcommand{\ZZ}{\mathbb{Z}}

\newcommand{\PP}{\mathbb{P}}

\newcommand\defn[1]{{\bf #1}}

\renewcommand\AA{{\mathbb A}}

\newcommand\into{\hookrightarrow}

\newcommand\lie[1]{{\mathfrak #1}}

\newcommand\union{\cup}
\newcommand\Union{\bigcup}

\newcommand\dom{{\backslash}}
\newcommand\iso{\mathrel{\cong}}

\newcommand\calY{{\mathcal Y}}

\newcommand\rank{{\mathrm rank}}

\newcommand{\mapsfrom}{\mathrel{\reflectbox{\ensuremath{\mapsto}}}}

\begin{document}

\title[A Bruhat atlas for the Mehta--van der Kallen stratification
  of $T^* GL_n/B$]{A Bruhat atlas for the\\ Mehta--van der Kallen stratification
  of $T^* GL_n/B$}
\author{Allen Knutson}
\address{Department of Mathematics, Cornell University, Ithaca, NY}
\email{\href{mailto:allenk@math.cornell.edu}{allenk@math.cornell.edu}}
\urladdr{\url{http://math.cornell.edu/~allenk/}}
\thanks{AK was partially supported by NSF DMS-1700372.}

\author{Steven V Sam}
\address{Department of Mathematics, University of California, San Diego, CA}
\email{\href{mailto:ssam@ucsd.edu}{ssam@ucsd.edu}}
\urladdr{\url{http://math.ucsd.edu/~ssam/}}
\thanks{SS was partially supported by NSF DMS-1812462.}

\date{October 11, 2021}

\begin{abstract}
  Mehta and van der Kallen put a Frobenius splitting on
  the type A cotangent bundle $T^* GL_n/B$, thereby defining a
  stratification by compatibly split subvarieties,
  and they determined a few of the elements of this stratification.
  We embed $T^* GL_n/B$ as a stratum in a larger stratified
  (and Frobenius split) space $GL_n/B \times Mat_n$
  whose stratification we determine,
  thereby giving a full description of the one of Mehta--van der Kallen.
  The main technique is to endow $GL_n/B \times Mat_n$ with
  a {\em Bruhat atlas}, covering it with open sets that are 
  stratified-isomorphic to Bruhat cells (in $GL_{2n}/B_{2n}$).
  Among the consequences are that each stratum closure is
  normal and Cohen--Macaulay.
\end{abstract}

\maketitle

\setcounter{tocdepth}{1}
{\footnotesize \tableofcontents}

\section{Statement of results}\label{sec:intro}
 
The \defn{Bruhat decomposition} of the flag manifold $Fl(n)$
is induced from the $n^2$ statistics
$$ F^\bullet \mapsto \dim (F^j \cap E^i),\quad i,j\in [n] $$
where $(E^i)_{i\in [n]}$ (respectively $(E_i)_{i \in [n]}$)
denotes the standard flag on affine space $\AA^n$,
invariant under the upper triangular matrices $B = B_n \leq GL_n$
(respectively the anti-standard flag, invariant under the
lower triangular matrices $B^-$).
If we identify $Fl(n) \iso GL_n/B$ as usual, we can collect these
statistics into the single map
$$ gB/B\quad\mapsto\quad B g B\quad \in B\backslash GL_n/B \iso S_n $$
which takes the closure relation on strata to the (opposite)
Bruhat order on $S_n$.

In this paper we introduce and study a similar story on $Fl(n) \times Mat_n$,
where $Mat_n$ denotes the space of $n \times n$ matrices:

\begin{Proposition}\label{prop:package}
  The map
  \begin{eqnarray*}
    v \colon Fl(n) \times Mat_n &\to& B^-_{2n} \backslash GL_{2n} / B_{2n} \iso S_{2n} \\
    (gB/B, X) &\mapsto& 
  B^-_{2n} \begin{bmatrix}  X & g \\ w_0 g^{-1} & 0\end{bmatrix} B_{2n}
  \end{eqnarray*}
  is well-defined (independent of the choice of lift of $gB/B$ to $g$),
  and both determines and is determined by the $4n^2$ statistics

  $    \rank(E^i \into \AA^n \stackrel{X}{\to} \AA^n \onto \AA^n/E_j), \hfill
\rank(\AA^n \xrightarrow{X} \AA^n \onto \AA^n/(E_{n-i} + F^j)) - \dim (E_{n-i} \cap F^j), \quad$

$
    \rank(F^{n-i} \cap E^j \into \AA^n \xrightarrow{X} \AA^n) -  \dim (E^j \cap F^{n-i}), \hfill
    \rank(F^i \into \AA^n \stackrel{X}{\to} \AA^n \onto \AA^n/F^j).
\quad$
\end{Proposition}

Let $\calY$ be the poset of {\em closures} of the level sets of $v$.
Proposition \ref{prop:package} implies that $\calY$ 
bears a ranked embedding $\calY^{op} \into S_{2n}$, again called $v$.
Our motivating example of a $Y\in \calY$ is defined by the conditions
$$\rank(F^i \into \AA^n \stackrel{X}{\to} \AA^n \onto \AA^n/F^{i-1}) = 0, \qquad
i\in [n],$$
which are equivalent to $X(F^i) \leq F^{i-1}$. This is the ubiquitous
\defn{Springer space} $Y_{Spr}$ of type $A$, isomorphic to $T^* GL_n/B$,
whose projection $(F^\bullet,X)\mapsto X$ is the Springer resolution
of the nilpotent cone. Another motivating example is the
\defn{Grothendieck--Springer family} $Y_{GS}$ defined by $X(F^i) \leq F^i$.
Each of these inherit stratifications, which are described at
the end of Proposition \ref{prop:G-invt} below.

\subsection*{Example: $n=2$}
Below is drawn the poset $\calY$ of strata inside $\PP^1 \times Mat_2$,
indexed by their elements in $S_{2n} = S_4$. We have grouped them
according to their projections to $Mat_2$ (which are automatically
compatibly Frobenius split therein), whose equations are written
in the grouping;
additional equations involving $F^1 \in \PP^1$ are indicated
by the superscripts, legend at right.
We have drawn in blue only the Hasse diagram edges for those projections;
the full partial order is just the restriction of $S_4$ Bruhat order.

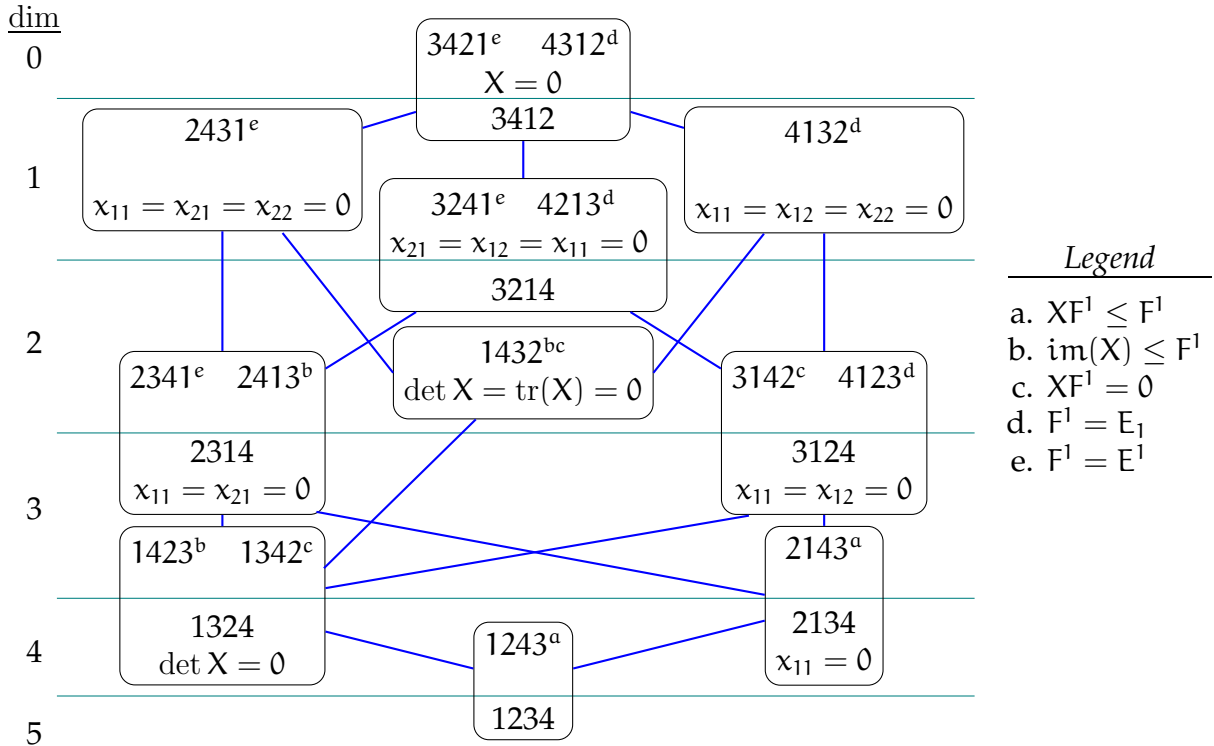
\begin{tikzpicture}
  \node at (-2.5,7) {$\underline{\dim}$};
  \node at (-2.5,6.5) {0};  \draw[color=teal] (-2.2,5.95) -- (10,5.95);
  \node at (-2.5,4.9) {1};  \draw[color=teal] (-2.2,3.8) -- (10,3.8);  
  \node at (-2.5,2.7) {2};  \draw[color=teal] (-2.2,1.5) -- (10,1.5);  
  \node at (-2.5,.5) {3};  \draw[color=teal] (-2.2,-.7) -- (10,-.7);
  \node at (-2.5,-1.4) {4};  \draw[color=teal] (-2.2,-2) -- (10,-2);    
  \node at (-2.5,-2.5) {5};
  \node[draw, align=center, rounded corners=.2cm] (A) at (4,6.2)
  {$3421^e$ \quad $4312^d$\\ $X=0$\\ 3412};
  \node[draw, align=center, rounded corners=.2cm] (B) at (0,5) 
  {$2431^{e}$\\ \\$x_{11}=x_{21} = x_{22}= 0$};
  \node[draw, align=center, rounded corners=.2cm] (C) at (4,4) 
  {$3241^{e}$\quad $4213^d$\\ $x_{21} = x_{12} = x_{11} = 0$ \smallskip \\ 3214};
  \node[draw, align=center, rounded corners=.2cm] (D) at (8,5) 
  {$4132^{d}$\\ \\$x_{11}=x_{12} = x_{22} = 0$};  
  \node[draw, align=center, rounded corners=.2cm] (E) at (0,1.5)
  {$2341^e$\quad $2413^b$\\ \\ 2314\\$x_{11}=x_{21} = 0$};
  \node[draw, align=center, rounded corners=.2cm] (F) at (4,2.3) 
  {$1432^{bc}$\\ $\det X={\rm tr}(X) = 0$};
  \node[draw, align=center, rounded corners=.2cm] (G) at (8,1.5)
  {$3142^c$\quad $4123^d$\\ \\ 3124\\$x_{11}=x_{12} = 0$};
  \node[draw, align=center, rounded corners=.2cm] (H) at (0,-.8) 
  {$1423^b$\quad $1342^c$\\ \\ 1324\\$\det X = 0$};
  \node[draw, align=center, rounded corners=.2cm] (I) at (8,-.8) 
  {$2143^a$\\ \\ 2134\\ $x_{11}=0$};
  \node[draw, align=center, rounded corners=.2cm] (J) at (4, -1.8) 
  {$1243^a$\\ \\ 1234};
  \draw[color=blue, thick] (A) -- (B) -- (E) -- (H) -- (J);
  \draw[color=blue, thick] ($(B)+(.8,-.84)$) -- ($(F)+(-1.73,0)$);
  \draw[color=blue, thick] (A) -- (C) -- (E);
  \draw[color=blue, thick] ($ (E) + (1.25,-1.05) $) -- (I);
  \draw[color=blue, thick] (A) -- (D) -- (G) -- (I) -- (J);
  \draw[color=blue, thick] ($(D) + (-.8,-.85)$) -- ($(F)+(1.73,0)$);
  \draw[color=blue, thick] (F) -- ($(H)+(1.35,.5)$);
  \draw[color=blue, thick] (H) -- ($(G)+(-1,-1.1)$);
  \draw[color=blue, thick] (C) -- (G);
  \node[
  align=left, text width=2.7cm, rounded corners=.2cm] at (11.8,2.5)
  {\centerline{ \em Legend }
    \hrule \vskip .1in
    \begin{enumerate}[a.]
    \item $XF^1 \leq F^1$ \\ 
    \item $im(X) \leq F^1 $\\ 
    \item $XF^1 = 0$ \\
    \item $F^1 = E_1$ \\ 
    \item $F^1 = E^1 $ 
    \end{enumerate}};
\end{tikzpicture}

For those familiar with the interplay of stratifications
under the projection $G/B \onto G/P$ (studied in \cite{projected}),
we mention two ways this projection $Fl(n)\times Mat_n \onto Mat_n$
is unlike that one:
\begin{enumerate}
\item Not every compatibly split subvariety of $Mat_{n}$
  is the image of a stratum in $Fl(n)\times Mat_n$; for example,
  $\{X\, :\, Tr(X)=0\}$ is not such an image.
\item For some compatibly split strata $Z$ in $Mat_n$ that {\em are}
  images of compatibly split strata $\widetilde{Z}$ in $Fl(n)\times Mat_n$,
  there don't exist $\widetilde{Z}$ mapping birationally to
  (or even of the same dimension as) $Z$.
\end{enumerate}

Of particular interest are the point strata $\calY_{min} \iso S_n$,
of the form $(\pi B/B, 0)$, $\pi\in S_n$. Restricted to those,
our map $v$ gives an injection $S_n\to S_{2n}$ yet again denoted $v$,
with
\[
  \ell(v(\pi)) = \dim(Fl(n)\times Mat_n) = {n\choose 2} + n^2
\]
for all $\pi\in S_n$.

\begin{Theorem} \label{thm:main}
  The following properties hold for the poset $\calY$ of strata: 
  \begin{enumerate}
  \item This $\calY$ is a \defn{stratification by closed subvarieties},\footnote{%
      It is more usual to axiomatize stratifications using disjoint
      locally closed subvarieties, but (except for the reducedness
      requirement) this approach is equivalent and generally more convenient.}
    i.e., each $Y\in \calY$ is irreducible, and the scheme-theoretic intersection $Y_1\cap Y_2$ of
    two subvarieties $Y_1,Y_2 \in \calY$ is reduced and is a union of other $Z\in \calY$.
  \item 
    The image of $v\colon  \calY^{op} \into S_{2n}$ is an order ideal
    in the Bruhat order of $S_{2n}$, with maximal elements $\{v(\pi) \in S_{2n}\, :\, \pi \in S_n\}$.
  \item For $\pi \in S_{2n}$, let $\Gamma = \{\gamma \leq \pi\}$ be the maximal
    biGrassmannian elements below $\pi$, computable (as in \cite{kobayashi})
    from the ``Fulton essential set'' of $\pi$.
    Then $Y_\pi$ is the scheme-theoretic intersection
    $\bigcap_{\gamma\in \Gamma} Y_b$, i.e., to find the equations
    for a general stratum it is enough to understand the biGrassmannian case.
  \item After reducing these schemes (defined over $\ZZ$) modulo any prime $p$,
    there is a (unique) Frobenius splitting on $Fl(n)\times Mat_n$
    with respect to which $\calY$ consists of
    exactly the compatibly split subvarieties.
  \item Identify $Fl(n) \iso GL_n/B$, so $Fl(n)$ bears an open cover
    $\Union_{\pi\in S_n} \pi B_- B/B$ by translates of the big cell.  There is
    a ``Bruhat atlas'' on $Fl(n) \times Mat_n$ identifying
    $$ \pi B_-B/B \times Mat_n \iso X^{v(\pi)}_\circ
    := B_{2n} v(\pi) B_{2n}/B_{2n} \qquad\subseteq GL_{2n}/B_{2n} \qquad\quad\ \ $$
    as stratified spaces, i.e.,  taking
    \[
      Y_\rho \cap (\pi B_-B/B \times Mat_n)
      \stackrel{\sim}{\to} X^{v(\pi)}_\circ \cap X_\rho\quad \text{for all } \rho\leq v(\pi)\text{, where }X_\rho := \overline{B^-_{2n} \rho B_{2n}}/B_{2n}.
    \]
    Ergo, every stratum in $\calY$ is normal and Cohen--Macaulay, with
    rational singularities.
  \end{enumerate}
\end{Theorem}

In particular (5) implies (1), (2), (4), and (4) implies (3).
So (5) is really the key statement. We recall the definition of
Bruhat atlases in \S \ref{sec:bruhatlas}.

There had not previously been a full determination of the strata in
this stratification of $T^* GL_n/B$; see \cite{MvdK,MvdK2} for examples.
Other Frobenius splittings of cotangent bundles of flag varieties were
studied in \cite{KLT,Hague}.

The strata thus constructed are each invariant under the action of 
the subgroup $B$ (by standard action on $F^\bullet$ and conjugation on $X$).
Our other result characterizes those strata which are
$GL_n$-invariant. Recall that the \defn{diagram of a permutation}
$\sigma \in S_m$ is the set
\[
  D(\sigma) := \{(i,j) \in [m] \times [m] \mid \sigma(i)>j, \ \sigma^{-1}(j)>i\}
\]
and the \defn{Fulton essential set} of $\sigma$ is the Southeast corners
of $D(\sigma)$, i.e. the set of matrix entries
$(i,j)\in D(\sigma)$ such that $(i+1,j) \notin D(\sigma)$ and
$(i,j+1) \notin D(\sigma)$.

\begin{Proposition} \label{prop:G-invt} Let $Y$ be a stratum, let
  $\sigma=v(Y)$, and let $\rho$ be the partial permutation of size $n$
  given by the bottom right $n \times n$ submatrix of $\sigma$. The
  following are equivalent:
  \begin{enumerate}
  \item $Y$ is $GL_n$-invariant.
  \item The Fulton essential set of $\sigma$
    lies in the bottom right $(n+1) \times (n+1)$ submatrix.
  \item Neither $\sigma$ nor $\sigma^{-1}$ have descents in $\{1,\dots,n\}$.
  \item We have an isomorphism $GL_n \times^B \overline{B \rho^T w_0 B} \to Y$
    where the closure of $B \rho^T w_0 B$ is taken inside of $Mat_n$.
  \end{enumerate}
  Moreover, each $n\times n$ partial permutation $\rho$ 
  arises from a unique such (full) permutation $\sigma \in S_{2n}$.

  Two examples are of special note:
  $$
  v(Y_{GS}) = \underline{1\ 2 \cdots n\ 2n\ (2n-1)\cdots\ (n+1)} \quad
  v(Y_{Spr}) = \underline{1\ 2\ 3\cdots (n-1)\ 2n\ (2n-1) \cdots n}
  $$
  whose induced stratifications match those defined by
  Mehta and van der Kallen in \cite{MvdK} (recalled in \S \ref{sec:strat}).
\end{Proposition}

\begin{remark}
  It would be of great interest if the conormal varieties to Schubert
  varieties could be made compatibly split with respect to some
  splitting on the cotangent bundle. Their union (inside $Y_{Spr}$) is
  given by the moment map equations ``the Southwest triangle of $X$
  vanishes''.  Unfortunately the conditions we impose on $X$ are {\em
    Northwest} rank conditions, so we have nothing to say about the
  conormal varieties.
\end{remark}

We comment briefly on the approach that suggested this space and this
Bruhat atlas. For a very long time we tried to cover $T^* GL_n/B$
itself with charts $X^\rho_\circ$ from the flag variety of a Kac--Moody group $H$.
This would require embedding the Mehta--van der Kallen poset as an
order ideal inside $W_H$. As far as we could tell, {\em that} would require
that $H$ have a rather fearsome Dynkin diagram, a sort of ``broom'' with
handle of length $n-2$ plus $n+1$ bristles
(nicely indexed by the covers of $v(Y_{Spr})$). It seems very hard to compute the defining equations 
of strata inside such general Kac--Moody flag varieties;
the closest work seems to be \cite{ElekLu}.
Much more recently we tried instead to embed the M--vdK poset (for small $n$)
as an order ideal inside $\{\rho \in S_m\colon \rho \geq \sigma\}$
for some fixed $\sigma$ with
$X_\sigma$ smooth (so that each $X^\rho_\circ \cap X_\sigma$ is a cell).
This succesful approach led to the discovery of $v(Y_{Spr})$.
That guess in turn prompted the
question of whether the ambient variety should be chosen larger than
$T^* GL_n/B$, larger by dimension $\ell(v(Y_{Spr})) = {n+1\choose 2}$,
and there was an obvious such choice.

\section{Stratifications and Frobenius splittings}
\label{sec:strat}

We start with three equivalent axiomatizations of ``stratification of $M$
by irreducible subvarieties''. The first and most traditional is a
finite disjoint decomposition $M = \coprod_{Y^\circ \in \calY^\circ} Y^\circ$
with the property that $\overline{Y^\circ} 
= \coprod_{Z^\circ \in \calY^\circ, Z^\circ \subseteq \overline{Y^\circ}} Z^\circ$.
We avoid using this definition because it is more pleasant to work
with closed subschemes rather than with locally closed subschemes.

The second, which is the one that we will use, is a finite collection $\calY$ of
subvarieties of $M$ with the property that
\[
  A,B\in \calY \implies A\cap B = \bigcup_{Z \in \calY,\, Z \subseteq A\cap B} Z.
\]
Already there is a subtlety compared to the first definition: should these intersections be required to be reduced?

We mention a third axiomatization of stratification: this is a finite
collection of subschemes $\calY^\union$ which is closed under union,
intersection, and taking geometric components.  This definition has
more of an ``algebra'' flavor, with two binary operations and a
(multi-valued) unary operation. One benefit of this point of view is
to allow for the definition of \defn{the stratification generated by a
  hypersurface} $H$, which amounts to decomposing $H$ into its
components, intersecting them with one another, and repeating those
two operations until one stops finding new varieties. There
is a new subtlety about reducedness, visible in the stratification of the
plane generated by the hypersurface $xy(x+y)=0$: should the
intersections of {\em unions} of strata be required to be reduced?

For the purposes of this paper, we will consider the most restrictive
form of these definitions for $\calY$ and require that any
intersection of closed subschemes in this stratification is reduced.
This will arise naturally if $M$ is endowed with a {\em Frobenius
  splitting}, whose definition we now recall from \cite[\S 1.3.1]{BK}.

If $R$ is a commutative algebra over $\mathbb{F}_p$, then a function
$\phi \colon R \to R$ is a \defn{Frobenius splitting} if, for all
$a,b \in R$, 
\begin{enumerate}
\item $\phi(a+b)=\phi(a)+\phi(b)$,
\item $\phi(a^p b) = a \phi(b)$, and
\item $\phi(1) = 1$.
\end{enumerate}

If $R$ admits a Frobenius splitting, then $R$ is easily seen to be
reduced: if $x \in R$ is nilpotent and non-zero, then let $m$ be the
largest power such that $x^m \ne 0$. Then $m \ge 1$ and $(x^m)^p=0$ and so
$0 = \phi(0) = \phi((x^m)^p) = x^m \phi(1) = x^m$, a contradiction.
One may extend $\phi$ uniquely to any localization $R[b^{-1}]$ by
$\phi(a/b) = \phi(a b^{p-1}/b^p) = \phi(a b^{p-1})/b$,
and this locality allows us to define splittings of schemes, not just rings
($\Longleftrightarrow$ affine schemes).

An ideal $I \subset R$ is \defn{compatibly split} if
$\phi(I) \subseteq I$. In that case, $\phi$ descends to a Frobenius
splitting on $R/I$ (hence $R/I$ is reduced, so $I$ is radical).
A nontrivial theorem (in what is still only a 3-page paper
\cite{KuM}, see also \cite{Schwede})
is that in a Frobenius split scheme of finite type,
the number of compatibly split subvarieties is finite.
It is easy to show they form a stratification in the most restrictive
of the three senses.

If $M$ and $H$ are both defined over $\ZZ$,
then one can transfer the corresponding reducedness results
to characteristic zero by standard spreading-out techniques.

In fact, we won't need to work carefully with this definition of $\phi$:
as we will see, the existence of a Bruhat atlas (defined in the next
section) guarantees that our space $M$ has a chart of affine spaces
such that the induced stratification on each is generated by a
hypersurface in the sense just explained.

In our situation, we can describe the irreducible components of the divisor of the stratification using the map $\nu$. More specifically, for $i=1,\dots,n$, the divisor labeled by $s_i \in S_{2n}$ is
\[
  Y_{s_i} = \{(F^\bullet, X) \mid \det(\text{NW $i \times i$ submatrix of $X$})=0\}.
\]
For $i=1,\dots,n-1$, the divisor labeled by $s_{n+i} \in S_{2n}$ is
\[
  Y_{s_{n+i}} = \{(F^\bullet, X) \mid \rank(F^{n-i} \xrightarrow{X} \AA^n/F^i) < n-i \}.
\]
In particular, every stratum can be obtained by the process of intersecting these divisors, taking irreducible components, intersecting with more divisors, and repeating.

Since $Y_{GS}$ is contained in each of the divisors of the second type, the induced stratification on $Y_{GS}$ that we get is generated by the intersections $Y_{GS} \cap Y_{s_i}$ for $i=1,\dots,n$. This matches with the stratification defined by \cite{MvdK}.

\section{The Bruhat atlas}\label{sec:bruhatlas}

We recall the definition of \emph{Bruhat atlas} due to X. He,
J.-H. Lu, and the first author \cite{Elek,GKL,H1,H2,LY,BH}.
A \defn{Bruhat chart} on a space $M$ with stratification $\calY(M)$
is a triple $(H,v_0,c)$ where
\begin{enumerate}
\item $H$ is a Kac--Moody group with standard Borel subgroups
  $B^\pm_H$ and Weyl group $W_H$,
\item $v_0$ is an element of $W_H$; it has an associated
  finite-dimensional \defn{Bruhat cell}
  \[
    X^{v_0}_\circ := B_H v_0 B_H / B_H
  \]
  of dimension $\ell(v_0)$, which is stratified by its
  intersection with $B^-_H$-orbit closures
  \[
    X_w := \overline{B^-_H w B_H}/ B_H,
  \]
\item $c \colon X^{v_0}_\circ \into M$ is an open embedding, and an isomorphism
  of stratified spaces with its image, i.e., the strata inside
  $c(X^v_\circ)$ are exactly the subsets $c(X^v_\circ \cap X_w)$ for $w\leq v$.
\end{enumerate}

In particular, each stratum in $M$ meeting the open image of $c$
receives a label $w\leq v_0$. If we can cover $M$ using charts, with
the labelings compatible across charts, we call this a \defn{Bruhat atlas}.
Unwrapping this discussion, a Bruhat atlas on $(M, \calY(M))$
is a triple
\[
  (H, v, \{c_p\}_{p \in \calY(M)_{min}})
\]
where
\begin{enumerate}
\item $H$ is a Kac--Moody group with standard Borel subgroups $B^\pm_H$
  and Weyl group $W_H$,
\item $\calY(M)_{min}$ is the set of minimal strata, each a point,
\item $v\colon \calY(M)^{op} \into W_H$ is a ranked embedding of posets,
  with image $\bigcup_{p \in \calY(M)_{min}} [1, v(p)]$,
\item for each $p \in \calY(M)_{min}$ we have an open embedding
  $c_p\colon  X^{v(p)}_\circ \into M$ which is an isomorphism (to its image) of
  $W_H$-stratified spaces. That is, for each stratum $Y \in \calY$,
  $c_p$ restricts to an isomorphism
  \[
    X^v_\circ \cap X_{v(Y)}  \xrightarrow{\cong} c_p(X^v_\circ) \cap Y,
  \]
\item these open images $\{c_p(X_\circ^{v(p)})\}$ cover $M$.
\end{enumerate}

This concept (foreshadowed in \cite{Snider,KWY}) serves as a very
efficient organizing principle for describing a stratified space and
properties of its strata. Here is a sample:

\begin{Proposition}\label{prop:atlas}
  Suppose that $(M, \calY(M))$ has a Bruhat atlas.
  \begin{enumerate}
  \item Each $Y \in \calY(M)$ has rational singularities (and, in
    particular, is normal and Cohen--Macaulay).
  \item The stratification $\calY(M)$ is generated by its hypersurface,
    i.e. it is the coarsest stratification by closed subvarieties
    (as defined in Theorem \ref{thm:main}(1)) that includes
    the components of the complement of the open stratum.

    Indeed, the induced stratification on any stratum $Y \in \calY(M)$
    likewise is generated by {\em its} hypersurface 
    $\bigcup\, \{ Z \in \calY(M)\colon Z \subsetneq Y\}$.
  \end{enumerate}
\end{Proposition}

\begin{proof}
  \begin{enumerate}
  \item This is true of each $X^v_\circ \cap X_w$,
    hence it is true of the open cover $\{ c_p(X^{v(p)}_\circ \cap X_{v(Y)}) \}$
    of $Y$, and these conditions are local so checkable on an open cover.
  \item Since the charts cover, each $Y\in \calY(M)$ meets some chart
    $(\pi N_- B/B)\times Mat_n$ in some open set $Y'$. 
    As seen in the proof of \cite[Theorem 2.3.1]{BK}, the stratification on
    any chart $X^{v(\pi)}_\circ$ has this generated-by-hypersurface property.
    The same proof, based on a combinatorial property of Bruhat order,
    extends to strata $X_w \cap X^{v(\pi)}_\circ$.
    \qedhere
  \end{enumerate}
\end{proof}

Concretely, this generation consists of taking the known
    compatibly split subvarieties, intersecting them with one another,
    picking out the components, and repeating until done.
    If we have a sequence of such operations on strata in
    $(\pi N_- B/B)\times Mat_n$ culminating in the stratum $Y'$ of
    $(\pi N_- B/B)\times Mat_n$, we can do the same with their closures
    in $Fl(n)\times Mat_n$, culminating in a union of $Y$ possibly with some
    other components (outside the chart). From there we pick out,
    and have ``generated'', $Y$. 

Fix $\pi\in S_n$ for the remainder of the section.
We recall \cite[Lemma 2.1]{KWY}:

\begin{Lemma}\label{lem:factor}
  Let $E^\pm := (\pi N^- \pi^{-1}) \cap N^\pm$. Then each $g \in \pi N^- \pi^{-1}$ can be uniquely factored
  as $g = b_- b_+ = c_+ c_-$ with $b_\pm,c_\pm \in E^\pm$, and the map
  \begin{eqnarray*}
    \Phi\colon \pi N^- \pi^{-1} &\to& E^+ \times E^- \\
    g &\mapsto& (b_+,c_-) 
  \end{eqnarray*}
  is an isomorphism of varieties.
\end{Lemma}

\begin{proof}[Proof of Theorem~\ref{thm:main}(5).]
  We put co\"ordinates on $X^{v(\pi)}_\circ$, first as the free orbit of the group
  $$ N_{2n} \cap (v(\pi) N^-_{2n} v(\pi)^{-1})
  = \left\{    \begin{bmatrix} a & Y \\ 0 & d    \end{bmatrix} \ :\
    a \in N \cap \pi N^- \pi^{-1},\
    d \in N \cap w_0 \pi^{-1} N^- \pi w_0 \right\}
  $$
  through $v(\pi)B_{2n}/B_{2n}$, taking
  \begin{eqnarray*}
     \begin{bmatrix} a & Y \\ 0 & d    \end{bmatrix} \mapsto
  \begin{bmatrix} a & Y \\ 0 & d    \end{bmatrix} 
  \begin{bmatrix} 0 & \pi \\ w_0 \pi^{-1} & 0    \end{bmatrix} B_{2n}/B_{2n}
  &=&
  \begin{bmatrix} Y w_0\pi^{-1} & a\pi \\ d w_0 \pi^{-1} & 0    \end{bmatrix}
  B_{2n}/B_{2n} \\
  &=&
  \begin{bmatrix} Y w_0\pi^{-1} & \pi a' \\ w_0 d' \pi^{-1} & 0    \end{bmatrix}
  B_{2n}/B_{2n} 
  \end{eqnarray*}
  where $a' \in N^- \cap \pi^{-1}N \pi$, $d' \in N^- \cap \pi^{-1}N^-\pi$ are
  the co\"ordinates we will actually use.

  Meanwhile we can factor
  \begin{equation}\label{eq:factor}
    \begin{array}{rcl}
      \begin{bmatrix} X& \pi g \\ w_0 g^{-1} \pi^{-1} & 0\end{bmatrix}
  &=&
  \begin{bmatrix} \pi b_- \pi^{-1} & 0 \\ 0 & I\end{bmatrix}
  \begin{bmatrix} X'& \pi b_+ \\ w_0 g^{-1} \pi^{-1} & 0\end{bmatrix} \\ \\
  &=&
  \begin{bmatrix} \pi b_- \pi^{-1} & 0 \\ 0 & I\end{bmatrix}
  \begin{bmatrix} X''& \pi b_+ \\ w_0 c_-^{-1} \pi^{-1} & 0\end{bmatrix} 
  \begin{bmatrix} \pi c_+^{-1} \pi^{-1} & 0 \\ 0 & I\end{bmatrix}
    \end{array}
    \tag{$\ast$}
  \end{equation}
  where $b_- b_+ = c_+ c_- = g$ are the unique factorizations of $g\in N^-$ from
  Lemma \ref{lem:factor}, and $X'' = \pi b_-^{-1} \pi^{-1} X \pi c_+ \pi^{-1}$.
  We then have maps
  $$
  \begin{array}{cccrl}
  c_\pi^{-1}\colon& (\pi g, X) &\mapsto& 
  \begin{bmatrix} \pi b_-^{-1} \pi^{-1} X \pi c_+ \pi^{-1}
    & \pi b_+ \\ w_0 c_-^{-1} \pi^{-1} & 0\end{bmatrix} B_{2n}/B_{2n}
    \\ \\
                  & \left(\pi\, \Phi^{-1}(a',d'^{-1}),
                    \pi b_- \pi^{-1} Z \pi c_+^{-1} \pi^{-1}  \right)   &\mapsfrom&
  \begin{bmatrix} Z & \pi a' \\  w_0 d' \pi^{-1} & 0    \end{bmatrix}
                                                           B_{2n}/B_{2n}
  &\colon c_\pi    
  \end{array}
  $$
  giving inverse isomorphisms
  of $(\pi\, N^-B/B)\times Mat_n$ with $X^{v(\pi)}_\circ$,
  where in the second line $b_- := \Phi^{-1}(a',d'^{-1}) a'^{-1}$,
  $c_+ := \Phi^{-1}(a',d'^{-1}) d'$.
  
  To see that these are {\em stratified} isomorphisms, we have to compare
  the matrix $c_\pi^{-1}(\pi g, X)$ to the matrix
  $  \begin{bmatrix} X& \pi g \\ w_0 g^{-1} \pi^{-1} & 0\end{bmatrix}$
that  we used in \S\ref{sec:intro} to define the stratification.
  This is exactly accomplished in (\ref{eq:factor}), showing that the
  two lie in the same double coset $B_{2n}^- \dom GL_{2n} / B_{2n}$.  
\end{proof}

\section{The remaining proofs}\label{sec:proofs}

\newcommand{\tikzmarkb}[1]{\tikz[overlay,remember picture] \node (#1) {};}
\newcommand{\DrawBox}[3][]{     \tikz[overlay,remember picture]{
      \draw[red,#1] ($(left)+(#2 em,0.9em)$) rectangle
      ($(right)+(#3 em,.3 em)$);}  } 
        
\begin{proof}[Proof of Proposition~\ref{prop:package}.]
  Each $(i,j) \in [2n]^2$ determines a flush-Northwest submatrix
  with Southeast corner $(i,j)$, indicated in red in the matrices below,
  and a complete set of $B^-_{2n} \times B_{2n}$-invariants on $GL_{2n}$ is
  given by the ranks of these submatrices.  
  The four types of statistics correspond to the four quadrants
  in the $2n\times 2n$ matrix $v(gB/B,X)$.
  We write $V$ for $\AA^n$ in this proof, to better distinguish it
  from $V^*$.
  Suppose that $g$ represents the flag $F^\bullet \subset V$, i.e.,
  $F^i$ is the span of the first $i$ columns of $g$.
  The first $i$ rows of $w_0 g^{-1}$ (thought of as functionals on $V$
  in the natural way) span $ann(F^{n-i}) \leq V^*$. 

  \begin{minipage}{.17\linewidth}
    $
    \left[\begin{array}{*{13}{c}}        \tikzmarkb{left} \tikzmarkb{right}
            X & g \\ w_0 g^{-1} & 0  \end{array}\right] 
        \DrawBox[thick]{.4}{-1.2}
        $
      \end{minipage}
      \hfill
  \begin{minipage}{.8\linewidth}
    \parindent = 1em
    The $(i,j)$ in the Northwest quadrant are simplest, giving Northwest
    rank conditions on $X$ itself. These compute the first set
    of numbers, $\rank(E^i \into \AA^n \stackrel{X}{\to} \AA^n \onto \AA^n/E_j)$.
  \end{minipage}

  \begin{minipage}{.17\linewidth}
    $    \left[\begin{array}{cc}            \tikzmarkb{left}
            X & g \tikzmarkb{right} \\ w_0 g^{-1} & 0 \end{array}\right]
        \DrawBox[thick]{-1.2}{-.2}$
  \end{minipage}
  \hfill
  \begin{minipage}{.8\linewidth}
    \parindent = 1em
    Pick $i \le n$ and consider the NW $i \times (n+j)$
    submatrix. Then the rank of this matrix is
    $\rank(X \colon V \to V/(E_{n-i} + F^j))$ plus the rank of the
    last $j$ columns. This last quantity is the rank of the map
    $F^j \to V/E_{n-i}$, which is $j - \dim(F^j \cap E_{n-i})$.
  \end{minipage}

\begin{minipage}{.17\linewidth}
$ \left[\begin{array}{*{13}{c}}            \tikzmarkb{left}
            X & g \\ w_0 g^{-1} \tikzmarkb{right} & 0 \end{array}\right]
\DrawBox[thick] {-1.2}{-1.2}
$
\end{minipage}
  \hfill
  \begin{minipage}{.8\linewidth}
    \parindent = 1em
Note that
  $(V^*/ann(F^{i}))^* = F^{i}$. So by taking transpose, the rank of
  the NW $(n+i) \times n$ matrix is
  \[    i + \rank(X^* \colon V^* \to V^*/ann(F^{n-i})) = i + \rank(X
    \colon F^{n-i} \to V).  \]
  Similar remarks as above allow us to recover the rank of the NW
  $(n+i) \times j$ matrix from knowledge of
  $\rank(X \colon F^{n-i} \cap E^{j} \to V)$ and
  $\rank(id \colon E^j \to V/F^{n-i})$.

\end{minipage}

\vskip .1in

\begin{minipage}{.35\linewidth}
$\left[\begin{array}{*{13}{c}}            \tikzmarkb{left}
            X & g \\ w_0 g^{-1} & 0 \tikzmarkb{right} \end{array}\right]
\DrawBox[thick] {-1.2}{-.1}
\sim
\left[\begin{array}{*{13}{c}}            \tikzmarkb{left}
            g^{-1} X g & I \\ w_0  & 0 \tikzmarkb{right} \end{array}\right]
\DrawBox[thick] {-.1}{-.1}
$
\end{minipage}
  \hfill
  \begin{minipage}{.62\linewidth}
    \parindent = 1em
  Finally, we want to understand the rank of the NW
  $(n+i) \times (n+j)$ submatrix. For that, we can use column
  operations and row operations to say that its rank is
  $i+j + \rank(X \colon F^{n-i} \to V/F^j)$ (the $i$ is the contribution
  from the last $i$ rows, the $j$ is the contribution from the last
  $j$ columns, and the remainder is what happens when we reduce the NW
  $n \times n$ matrix). \qedhere
  \end{minipage}
\end{proof}

\begin{proof}[Proof of Proposition~\ref{prop:G-invt}]
  $(1) \Longrightarrow (2)$: The rank conditions given by the pairs in
  the essential set of $\sigma$ are non-redundant \cite[Lemma 3.14]{Fulton91},
  and hence none of them can be omitted. Now $Y$ is $G$-invariant if
  its defining conditions do not involve the standard flag $E$, and from the 
  proof of Proposition~\ref{prop:package}, this is equivalent to all 
  essential boxes being in the bottom right $(n+1) \times (n+1)$ submatrix. 
  This latter sentence also shows $(2) \Longrightarrow (1)$.

  $(2) \Longrightarrow (3)$: Suppose that $i$ is a descent of $\sigma$
  so that $\sigma(i)>\sigma(i+1)$. Then for some $j$ such that
  $\sigma(i+1) \le j < \sigma(i)$, the pair $(j,i)$ is in the
  essential set of $\sigma$.  In particular, if the essential set is
  contained in the bottom right $(n+1) \times (n+1)$ submatrix, then $i > n$. 
  Similarly, we see that $\sigma^{-1}$ has no descents amongst $1,\dots,n$.

  $(3) \Longrightarrow (2)$: 
  We address here the existence and uniqueness of $\sigma$ given $\rho$.
  Let $C,R \subseteq [n]$ be the set of columns
  and rows in $\rho$ that contain $1$s, so $|C|=|R|=\rank(\rho)$.
  Then it is easy to determine the rest of $\sigma$:
  \begin{equation} \label{eq:sigma}
    \sigma \quad=\quad 
    \kbordermatrix{ &|C|&n-|C|&&n \\
      |R|& I & 0 & \vline & 0 \\
      n-|R|& 0   & 0   &\vline& J \\ \hline
      n & 0 & K     &\vline& \rho    
    }
    \tag{$\ast\ast$}
  \end{equation}
  where $J$ is a partial permutation matrix with $1$s in columns
  $[n]\setminus C$ running NW/SE, and $K$ is similar, except with $1$s
  in rows $[n]\setminus R$ running NW/SE. From here we see that the
  diagram in the NW quadrant consists of the $(n-|C|)^2$ square in its
  SE corner, the diagram in the NE quadrant consists of columns touching
  the bottom row, and the diagram in the SW quadrant consists of rows 
  touching the right column.

  $(1,2,3) \Longrightarrow (4)$: 
  By $(1)$, the projection $Y\to Fl(n)$ is a $GL_n$-equivariant map to
  a homogeneous $GL_n$-space, hence determined by its fiber over $B/B$;
  specifically, the action map
  \[
    GL_n \times^B (Y\cap (\{B/B\} \times Mat_n ) \to Y
  \]
  is an isomorphism. We analyze $Y\cap (\{B/B\} \times Mat_n)$ 
  using $v:
   (B/B,X) \mapsto B^-_{2n} \begin{bmatrix}  X & I \\ w_0 & 0\end{bmatrix} B_{2n}.
  $
  Asking this to be in $\overline{B^-_{2n} \sigma B_{2n}}$ is
  (by (2) and \cite{Fulton91}) equivalent to asking that,
  for every $i,j \in [n]$, the rank of the NW $(n+i)\times (n+j)$ submatrix of
  $\begin{bmatrix} X & I \\ w_0 & 0\end{bmatrix}$ has rank bounded by that of 
  the corresponding submatrix of $\sigma$.
  Continuing \eqref{eq:sigma}, we break $\sigma$ into blocks as
  $$ 
  \kbordermatrix{ &|C|&n-|C|&&j & n-j\\
    |R|& I & 0 & \vline & 0 & 0\\
    n-|R|& 0   & 0   &\vline& J_1 & J_2 \\ \hline
    i & 0 & K_1     &\vline& \rho_1 & \rho_2      \\
    n-i & 0 & K_2    &\vline& \rho_3 & \rho_4 }
  \begin{array}{c}
    \text{with}\\ \text{ranks}
  \end{array}
  \kbordermatrix{ &|C|&n-|C|&&j & n-j\\
    |R|& |C| & 0 & \vline & 0 & 0\\
    n-|R| & 0   & 0   &\vline&  j-a-c & n-j-b-d \\ \hline
    i & 0 & i-a-b     &\vline& a & b      \\
    n-i & 0 & n-i-c-d    &\vline& c & d }.
  $$
  where $|C|=a+b+c+d$.
  Hence the NW $(n+i) \times (n+j)$ submatrix of $\sigma$ has rank 
  $$
  |C| + (j-a-c) + (i-a-b) + a  = (a+b+c+d) + (j-a-c) + (i-b) = d + j + i.
  $$
  Meanwhile, the corresponding submatrix 
  $  \begin{bmatrix}     X & {I_j\atop 0} \\ 0\ w_0^i & 0   \end{bmatrix}$
  has rank $i+j$ plus the rank of the SW $(n-j)\times (n-i)$ 
  submatrix of $X$.

  Together, $(B/B,X)$ is in $Y$ if and only if
  for each $i,j$,
  \begin{eqnarray*}
    rank(\text{SW $(n-j)\times (n-i)$ rectangle of }X)
    &\leq& rank(\text{SE $(n-i)\times (n-j)$ of }\rho) \qquad =d\\
    &=& rank(\text{SE $(n-j)\times (n-i)$ of }\rho^T) \\
    &=& rank(\text{SW $(n-j)\times (n-i)$ of }\rho^Tw_0)
  \end{eqnarray*}
  We have reached Fulton's equations for $\overline{B \rho^T w_0 B}$.

  $(4) \Longrightarrow (1)$: For any $B$-variety $Z$, $G \times^B Z$
  is always $G$-invariant.

  It remains to check the two examples, whose permutations we
  call $\sigma_{GS}$, $\sigma_{Spr}$. In each the southeast quadrant $\rho$
  of $v(Y)$ is easy to calculate, giving us the spaces $GL_n \times^B \lie{b}$ 
  and $GL_n \times^B \lie{n}$, respectively. Since (both here and 
  in \cite{MvdK}) the stratification of $Y_{Spr}$ is restricted from 
  the one on $Y_{GS}$, we need only check that we have the right
  stratification of $Y_{GS}$, which (thanks to Proposition \ref{prop:atlas} (2))
  is determined by its hypersurface.  

  The codimension $1$ strata inside $Y_{GS}$ correspond to the
  Bruhat covers $w \gtrdot \sigma_{GS}$. These come in two types:
  $\sigma s_i = s_i \sigma$ for $i<n$, and
  $\sigma \,\circ (n\leftrightarrow j)$ for $j \in [n+1,2n]$. The latter type are exactly the least upper bounds
  of $\sigma_{GS}$ and $s_{n}$. Consequently, the codimension $1$ strata
  come from the intersection of $Y_{GS}$ with $Y_{s_i}$ for $i\in [1,n]$.
  Since $Y_{s_i} = \{ (F^\bullet,X) \colon \det(\text{NW $i\times i$ minor of $X$}) = 0\}$,
  we have recovered the defining divisor from \cite[\S 3.4]{MvdK}.
\end{proof}

\end{document}